\documentclass{amsart}
\usepackage{amscd,amssymb,graphics}
\usepackage{hyperref}
\usepackage{amsfonts}
\usepackage{amsmath}
\usepackage{enumerate}
\usepackage{tikz}

\input xy
\xyoption{all}

\usepackage{vmargin}
\setmarginsrb{20mm}{10mm}{20mm}{13mm}
          {10mm}{7mm}{10mm}{10mm}

\newtheorem{theorem}{Theorem}[section]
\newtheorem{fact}[theorem]{Fact}
\newtheorem{corollary}[theorem]{Corollary}
\newtheorem{lemma}[theorem]{Lemma}
\newtheorem{proposition}[theorem]{Proposition}

\newtheorem{question}[theorem]{Question}
\newtheorem{definition}[theorem]{Definition}
\numberwithin{equation}{section}
\theoremstyle{remark}
\newtheorem{remark}[theorem]{Remark}
\newtheorem{example}[theorem]{Example}
\newtheorem{problem}[theorem]{Problem}

\newcommand{\ben}{\begin{enumerate}}
\newcommand{\een}{\end{enumerate}}

\def\R {{\mathbb R}}

\def\N{{\mathbb N}}

\def\Z {{\mathbb Z}}

\def\End{\operatorname{End}}
\def\Aut{\operatorname{Aut}}
\def\AT{\operatorname{AT}}
\def\cclctqh{compactly covered locally compact tqh}

\begin{document}
\title{Algebraic entropy on topologically quasihamiltonian groups}
\author[Xi]{W. Xi}
\address[W. Xi]
{\hfill\break School of Mathematical Sciences
\hfill\break Nanjing Normal University
\hfill\break Wenyuan Road No. 1, 210046 Nanjing
\hfill\break China}
\email{xiwenfei0418@outlook.com}

\author[Shlossberg]{M. Shlossberg}
\address[M. Shlossberg]
{\hfill\break Mathematics Unit
\hfill\break Shamoon College of Engineering
\hfill\break 56 Bialik St., Beer-Sheva 84100
\hfill\break Israel}
\email{menacsh@sce.ac.il}
		
\author[Toller]{D. Toller}
\address[D. Toller]
{\hfill\break Dipartimento di Matematica e Informatica
\hfill\break Universit\`{a} di Udine
\hfill\break Via delle Scienze  206, 33100 Udine
\hfill\break Italy}
\email{daniele.toller@uniud.it}

\keywords{{algebraic entropy, Addition Theorem, (topologically) quasihamiltonian group, Hamiltonian group, compactly covered group}}

\subjclass[2010]{37A35, 22D40, 28D20, 20K35}

\begin{abstract}
We study the algebraic entropy of continuous endomorphisms of compactly covered, locally compact, topologically quasihamiltonian groups. We provide a Limit-free formula which helps us to simplify the computations of this entropy. Moreover, several Addition Theorems are given. In particular, we prove that the Addition Theorem  holds for  endomorphisms of   quasihamiltonian torsion FC-groups (e.g.,  Hamiltonian groups).
\end{abstract}
	
\maketitle
%\today
%\tableofcontents
	
\section{Introduction}
The algebraic entropy was first considered in  \cite{AKM} for endomorphisms of (discrete) abelian groups.   Studying  this concept in the torsion case,   Weiss \cite{W}  connected  the algebraic entropy to the topological entropy (and also to the measure entropy)  using a Bridge Theorem. For a recent fundamental paper on the algebraic entropy for endomorphisms of torsion abelian groups we refer the reader to \cite{DGSZ}.

A different definition of the algebraic entropy was given by Peters \cite{Pet}. His definition was restricted to automorphisms of discrete abelian groups.
Note that these two definitions coincide on automorphisms of torsion abelian groups. Throughout the years there have been several extensions to Peters' entropy (see \cite{ DGBabelian, Pet1, V}).

Following Virili \cite{V} (see also \cite{DG-islam}) we give now the general definition for the algebraic entropy on (not necessarily abelian) locally compact groups.
Let $G$ be a locally compact group and  $\mu$ be a right Haar measure on $G$. For $\phi \in \End (G)$, a subset $U \subseteq G$, and $n\in\N_+$, the $n$-th $\phi$-trajectory of $U$ is
\[T_n(\phi,U) = U \cdot \phi(U)  \cdot\ldots\cdot  \phi^{n-1} (U).\]
When $U \in \mathcal C(G)= \{ \text{compact neighborhoods of } e_G \}$, the subset $T_n(\phi,U)$ is also compact, so it has finite measure.
The \emph{algebraic entropy of $\phi$ with respect to  $U$} is
\begin{equation}\label{def:H:alg}
H_{alg} (\phi, U) = \limsup_{n \to \infty} \frac{\log \mu (T_n(\phi,U))}{n},
\end{equation}
and it does not depend on the choice of the Haar measure $\mu$ on $G$.
The \emph{algebraic entropy of $\phi$} is
\begin{equation}\label{eq:algent}
h_{alg}(\phi) = \sup \{ H_{alg} (\phi, U) : U \in \mathcal C(G) \}.
\end{equation}
Note that $h_{alg}$ vanishes on compact groups.

In \cite{GBST}, Giordano Bruno and the last two authors studied the algebraic entropy on strongly compactly covered groups.  Recall that a topological group $G$ is \emph{compactly covered}  if every element of $G$ is contained in a compact subgroup. In case every element of $G$ is contained in a compact open normal subgroup of $G$ then $G$ is \emph{strongly compactly covered}. The discrete strongly compactly covered groups are FC-groups.  An \emph{FC-group}  is a group in which every element has finitely many conjugates. It is known that a torsion group is an FC-group if and only if each of its finite subsets is contained in a finite normal subgroup (see \cite[14.5.8]{R}); for this reason the torsion FC-groups are also called {\it locally finite and normal}. In particular, torsion FC-groups are locally finite.

In this paper, we study the algebraic entropy on compactly covered locally compact topologically quasihamiltonian groups (see Definition \ref{def:tqh}).
A discrete topologically quasihamiltonian group is quasihamiltonian (see Definition \ref{def:quasih}). In Section \ref{sec:tqh}, we study the topologically quasihamiltonian groups. In particular, we show (even in the discrete case) in Example \ref{example:nonFC} that a compactly covered locally compact topologically quasihamiltonian group is not necessarily strongly compactly covered. In fact, we prove in Proposition \ref{prop:pquafc} that a quasihamiltonian $p$-group is an FC-group if and only if it has finite commutator.

In Corollary \ref{cor:tqngen}, we generalize a measure-free formula given in \cite{DG-islam, GBST}. Using also Proposition \ref{prop:cof} we obtain in Equation (\ref{eq:alg})  a simplified formula for the algebraic entropy of an endomorphism of a compactly covered locally compact topologically quasihamiltonian group.

In Section \ref{basic-sec} we study the basic properties of the algebraic entropy for continuous endomorphisms of compactly covered locally compact topologically quasihamiltonian groups: Invariance under conjugation, Logarithmic Law and  Monotonicity for closed subgroups and Hausdorff quotients.

Moreover, under this setting the algebraic entropy of the identity automorphism is zero. Note that in general  the algebraic entropy of the identity automorphism need not vanish. In fact, the identity automorphism of a finitely generated group of exponential growth has infinite algebraic entropy  (see \cite{DG-islam,GBSp}).
We also find the precise relation between $h_{alg}(\phi)$ and $h_{alg}(\phi^{-1})$ by using the \emph{modulus of $\phi$} (see Proposition \ref{h:alg:phi:inverse}).
Inspired by \cite{CGB, GBV, WIL}, we extend in Proposition \ref{prop:limit-free} the Limit-free Formula given in \cite[Proposition 5.3]{GBST}.

\medskip

Let $G$ be a locally compact group, $\phi\in \End(G)$ and $H$ a closed normal $\phi$-invariant (i.e., satisfying $\phi(H)\leq H$)  subgroup of $G$. We say that the Addition Theorem holds for the algebraic entropy if
\begin{equation}\label{ATeq}
h_{alg} (\phi) =h_{alg} (\bar \phi) + h_{alg} (\phi \restriction_H),
\end{equation}
where $\phi \restriction_H$ is the restriction of $\phi$ to $H$ and $\bar \phi : G/H \to G/H$ is the induced map on the quotient group.

The Addition Theorem was proved in \cite[Theorem 3.1]{DGSZ} for the class of discrete torsion abelian groups, and  was later generalized  to the class of discrete abelian groups in \cite[Theorem 1.1]{DGBabelian}.

On the other hand, it is known \cite{GBSp} that the Addition Theorem does not hold in general for the algebraic entropy even for discrete solvable groups (while its validity for nilpotent groups is an open problem). This comes from the strict connection of the algebraic entropy with the group growth from Geometric Group Theory (see \cite{DG-islam,DG-pc,GBSp2}).

We consider the following general problem. Note that it is not even known whether the Addition Theorem holds in general for locally compact abelian groups (see \cite{DG-islam,DGB}).

\begin{problem}
For which locally compact groups does the Addition Theorem hold?
\end{problem}

In Section \ref{sec:AT} we prove several Addition Theorems. Our main result Theorem  \ref{addthm}, that should be compared with \cite[Theorem 7.1]{GBST}, shows that the Addition Theorem holds for a  compactly covered locally compact topologically quasihamiltonian group $G$  in case $H$ is a closed normal $\phi$-stable (i.e., $\phi(H)= H$) subgroup of $G$ with $\ker\phi\leq H$.
In particular, we find that the Addition Theorem holds for topological automorphisms of strongly compactly covered groups (see Corollary \ref{cor:foraut}).

In the discrete case this means that the Addition Theorem holds for automorphisms of torsion quasihamiltonian groups. Furthermore, the Addition Theorem holds even for continuous endomorphisms in case  the torsion quasihamiltonian groups are also FC-groups (see Theorem \ref{torquasiFC}).

Another consequence of our Addition Theorem (see Corollary \ref{c(G)}) is that, to compute the algebraic entropy of a topological automorphism $\phi$ of a compactly covered locally compact topologically quasihamiltonian group $G$, one can assume $G$ to be totally disconnected.

\vskip 0.3cm
Section \ref{Open questions and concluding remarks} collects some open questions and concluding remarks.
\subsection{Notation}
We denote by $\N$ the set of natural numbers, and by $\N_+$ and $\mathbb{P}$ its subsets of positive natural numbers and prime numbers, respectively.

Let $G$ be a group with identity $e_G$. For a non-empty subset $A$ of $G$, we denote by $\langle A \rangle$ the subgroup  generated by $A$.
If $x\in G$ and $\langle x \rangle$ is finite, then $x$ is  a \textit{torsion} element of $G$, and $o(x)=|\langle x \rangle|$ is the \textit{order} of $x$. For  $n\in \N_+$, let $G[n]=\{g\in G : g^n=e_G\}$ and $t(G)=\bigcup_{n\in \N_+}G[n]$ be the torsion part of $G$. In particular,  $G$ is \textit{torsion} if $t(G)=G$.  The group $G$ is called \textit{bounded of exponent $n$}, if the  least common multiple of the orders of all elements of $G$ is $n$.   For  $p\in \mathbb{P}$, the set $G_p=\{x\in G: \exists n\in \N \ \ o(x)=p^n \}$ is the \textit{$p$-component} of $G$.

 We denote by $G'$ the \textit{derived subgroup} of $G$, namely the subgroup of $G$ generated by all commutators $[a,b]=aba^{-1}b^{-1}$. As usual $Z(G)$ denotes the \textit{center} of $G$.

In this paper, we always consider Hausdorff topological groups. For a topological group $G$, we denote by $\End(G)$ the set of all continuous group endomorphisms of $G$ and by  $c(G)$ the \textit{connected component} of $G$.

The set of all compact subsets of $G$ is denoted by $\mathcal K(G)$, while $\mathcal C(G)$, $\mathcal B(G)$ and $\mathcal N(G)$ are the subfamilies of $\mathcal K(G)$ of all compact neighborhoods of $e_G$, all compact open subgroups and all compact open normal subgroups of $G$, respectively. Clearly, $\mathcal N(G)\subseteq\mathcal B(G)\subseteq\mathcal C(G)\subseteq\mathcal K(G)$.
	
\section{Topologically Quasihamiltonian groups}\label{sec:tqh}

In this section we study the topologically quasihamiltonian groups.
In the discrete case these groups are simply the quasihamiltonian groups (see Definition \ref{def:quasih}) that we study in \S \ref{QuasiHamilt}, then we dedicate \S \ref{gc} to the general (non-necessarily discrete) case.

\subsection{Quasihamiltonian groups}\label{QuasiHamilt}

We start  by recalling the following definition.
\begin{definition}\label{Hamilt}
A non-abelian group $G$ is called Hamiltonian if every subgroup of $G$ is normal.
\end{definition}
	
The following characterization theorem is due to Dedekind and Baer (see \cite[5.3.7]{R}).
\begin{fact}\label{thm:ham}
A group $G$ is  Hamiltonian if and only if $G \cong Q_8\times B \times D$, where $Q_8$ is the quaternion group of order $8$, $B$ is a Boolean group and $D$ is a torsion abelian group with all its elements of odd order.  In other words, $G \cong Q_8\times T$, where $T$ is an arbitrary torsion abelian group such that  $T_2 = T[2]$ is of exponent $\leq 2$.
\end{fact}

The next definition generalizes Definition \ref{Hamilt}.
\begin{definition}\label{def:quasih}
A group $G$ is called quasihamiltonian if for every pair of subgroups $X,Y$ of $G$ one has $XY=YX$ (equivalently, $XY$ is a subgroup of $G$).
\end{definition}
\begin{remark}\label{rem:per}
\  \ben
\item  A group $G$ is quasihamiltonian if it satisfies the above definition for every pair of \emph{cyclic} subgroups $X$ and $Y$.
\item If $G$ is a quasihamiltonian group and $X_1,X_2, \ldots, X_n$ are subgroups of $G$ with $n\geq 2$, then for every permutation $\pi\in S_n$ we have \[X_1\cdot X_2\cdot\ldots\cdot X_n=X_{\pi(1)}\cdot X_{\pi(2)} \cdot\ldots\cdot X_{\pi(n)}.\]
\item A torsion quasihamiltonian group is locally finite (this property will be generalized in Remark \ref{remark:closed}).
\een
\end{remark}
Iwasawa \cite{I} described the structure of quasihamiltonian groups  $G$ proving in particular that:

\begin{fact} \label{quasi}
If $G$ is a quasihamiltonian group, then
\ben
\item the torsion part $t(G)$ is a fully invariant \emph{subgroup} of $G$,  i.e., for every $\phi\in \End(G)$, the subgroup $t(G)$ is $\phi$-invariant;
\item $G$ is metabelian;
\item if $G$ is torsion-free, then $G$ is abelian;
\item if $G$ is mixed (i.e., $\{e\} \neq t(G) \neq G$), then both $t(G)$ and $G/t(G)$ are abelian. Moreover, if $G$ is non-abelian then $G/t(G)$ has rank $1$.
\een
\end{fact}

\medskip
It is easy to see from  Fact \ref{thm:ham} that the Hamiltonian groups are torsion, two-step nilpotent FC-groups.

Based on the following fact  we characterize in Proposition \ref{prop:pquafc} the quasihamiltonian $p$-groups which are also FC.

\begin{fact} \cite[Theorem 3]{I}(see also \cite[Theorem 18]{S})\label{fac:pgr}
Let $p$ be a prime number. A non-abelian $p$-group $G$ is quasihamiltonian if and only if $G$ is either Hamiltonian or $G$ contains an abelian normal subgroup $A$ with the following properties:
\ben
\item $A$ is bounded of exponent $p^n$, for $n \in \N_+$;
\item $G/A$ is a cyclic group of order $p^m$, for $m \in \N_+$;
\item there exists an element $t$ of $G$ and an integer $s\in \N_+$, such that $G=\langle A,t \rangle$, and $tat^{-1}=a^{1+p^s}$ for all $a\in A$,
$n\leq s+m, \  t^{p^{s+m}}=1$ and if $p=2$, then $s\geq 2$.
\een
\end{fact}

From Fact \ref{fac:pgr}(1)-(2) it follows that a non-abelian quasihamiltonian $p$-group $G$ is bounded, and its exponent divides $p^{n+m}$.
Moreover $s<n$, otherwise for every $a \in A$ we have $[t, a] = tat^{-1}a^{-1} = a^{p^s}$ from (3), so that $[t, a] =1$ from (1), and $G=\langle A,t \rangle$ would be abelian. So really $s < n\leq s+m$.

Using Fact \ref{fac:pgr} we can describe the commutator subgroup of a quasihamiltonian $p$-group.
\begin{lemma}\label{lem:der}
Let $G$ be a non-abelian quasihamiltonian $p$-group. In  the notation of Fact \ref{fac:pgr}, if $G$ is not Hamiltonian, then $G'=A^{p^s}$.
\end{lemma}
\begin{proof}
Since $tat^{-1}a^{-1}=a^{p^s}$ for all $a\in A$ we deduce that $A^{p^s}\leq G'$.

To prove the converse inclusion let $g_1=t^{r_1}a_1$, $g_2=t^{r_2}a_2$ be two arbitrary elements of $G$, where $a_1,a_2\in A$ and $r_1,r_2\in \N$. We need to show that the commutator $[g_1,g_2]=g_1g_2g_1^{-1}g_2^{-1}\in A^{p^s}$. We have
\begin{multline}
[g_1,g_2] = t^{r_1} a_1 t^{r_2} a_2 a_1^{-1} t^{-r_1} a_2^{-1} t^{-r_2}=
( t^{r_1} a_1 t^{-r_1} )( t^{r_1+r_2} a_2 t^{-r_1-r_2} )( t^{r_1+r_2} a_1^{-1}t^{-r_1-r_2} )( t^{r_2} a_2^{-1} t^{-r_2} )=\\
=a_1^{(1+p^s)^{r_1}} a_2^{(1+p^s)^{r_1+r_2}}  (a_1^{-1})^{(1+p^s)^{r_1+r_2}} (a_2^{-1})^{(1+p^s)^{r_2}}.
\end{multline}

As $p^s$ divides $(1+p^s)^{r}-1$ for every $r\in \N$, we deduce that
\[
[g_1,g_2]=
a_1 a_1^{p^s n_1 } \, a_2 a_2^{ p^s n_2} \, a_1^{-1} a_1^{-p^s n_2 } \, a_2^{-1} a_2^{-p^s n_3 } ,
\]
for some $n_1,n_2,n_3\in \N.$ Since $A$ is abelian we finally conclude
\[
[g_1,g_2]=(a_1^{n_1})^{p^s} (a_2^{n_2})^{p^s} (a_1^{-n_2})^{p^s} (a_2^{-n_3})^{p^s} \in A^{p^s}.\qedhere
\]
\end{proof}

\begin{proposition}\label{prop:pquafc}
Let $G$ be a  quasihamiltonian $p$-group. Then, $G$ is an FC-group if and only if $G'$ is finite.
\end{proposition}

\begin{proof}
 Every group with finite commutator is an FC-group (see \cite[Page 427]{R}), so we prove the  converse implication.

Since a Hamiltonian group is a torsion FC-group we may assume that $G$ is not Hamiltonian, as well that $G$ is not abelian.

So let $G$ be a non-abelian quasihamiltonian $p$-group that is an FC-group, and we have to show that $G'$ is finite. If $A\leq G$ is the subgroup described in Lemma \ref{lem:der}, it is equivalent to show that $A^{p^s}$ is finite.

For every $u\in A$ we have \[utu^{-1}=utu^{-1}t^{-1}t= u(u^{-1})^{1+p^s}t=(u^{-1})^{p^s}t.\] As $t$ has finitely many conjugates, this yields that the set
$\{(u^{-1})^{p^s}t: u\in A\}$ is finite. Clearly this is equivalent to the finiteness of the subgroup $A^{p^s}$.
\end{proof}

The following is an example of a torsion quasihamiltonian group that is not FC.

\begin{example}\label{example:nonFC}
Consider the action $\alpha:\Z(3)\times \Z(3^n)^{\N}\to \Z(3^n)^{\N}, \ \alpha(x,a)=(1+3^{n-1})^xa\bmod 3^n$, for some $n\geq 2$. Note that this action is well-defined. Indeed, if  $x\equiv y\bmod 3$ then $(1+3^{n-1})^x\equiv (1+3^{n-1})^y\bmod 3^n$. To see this, assume  without loss of generality that $x=y+3k$ for some $k\in \N.$
Then $(1+3^{n-1})^x-(1+3^{n-1})^y= (1+3^{n-1})^y((1+3^{n-1})^{3k}-1)$ and  clearly   $3^n$ divides $(1+3^{n-1})^{3k}-1$. Now let  $G=\Z(3^n)^{\N}\rtimes_{\alpha}\Z(3)$ be the $3$-group arising from the action $\alpha$.
	
By Fact \ref{fac:pgr}, $G$ is quasihamiltonian (here $p=3, A=\Z(3^n)^{\N}, n\geq2, m=1, s=n-1$ and $t=(0,1)$).
According to \ref{lem:der}, the commutator $G'=(\Z(3^n)^{\N})^{3^{n-1}}$ is infinite. So, $G$ is not FC by Proposition \ref{prop:pquafc}.
\end{example}

A torsion abelian group $G$ is the direct sum of its $p$-components $G_p$.  The following (apparently known) lemma extends this result to torsion quasihamiltonian groups. We prove it here for the sake of the reader.

\begin{lemma}\label{lem:dir}
Let $G$ be a torsion quasihamiltonian group. Then,
\ben
\item $G\cong \bigoplus_{p\in \mathbb P} G_p$;
\item $G'\cong \bigoplus_{p\in \mathbb P} G_p'$;
\item $G/G'\cong \bigoplus_{p\in \mathbb P}(G_p/G_p')$.
\een
\end{lemma}
\begin{proof}
(1) First, we show that $G_p$ is a subgroup of $G$ for every prime $p$. Clearly, if $x\in G_p$ then $x^{-1}\in G_p$. Taking $x,y\in G_p$ and using \cite[Equation (2)]{I}, we deduce that the cardinalities of $(\langle x\rangle\langle y\rangle)/\langle y\rangle$ and $\langle x\rangle/(\langle x\rangle\cap\langle y\rangle)$ are equal, so $xy\in G_p$.

Use \cite[Equation (2)]{I} again to deduce that every element in $G_{p_1}\cdot G_{p_2}\cdot\ldots\cdot  G_{p_k}$ is annihilated by a product of powers of $p_1,p_2,\ldots,p_k.$ Therefore, $G_{p}\cap (G_{p_1}\cdot G_{p_2}\cdot\ldots \cdot G_{p_k}) = \{e_G\}$, whenever $p \neq p_1,p_2,\ldots,p_k$. Thus the $G_p$'s generate the direct sum $\bigoplus_{p\in \mathbb P} G_p$ in $G$.

In order to show that every $x\in G$ lies in the direct sum, let $o(x) = m = p_1^{r_1}p_2^{r_2}\cdots p_n^{r_n}$ with different primes $p_1,p_2,\ldots, p_n$. For $i = 1,2,\ldots, n$,  the numbers $m_i = mp_i^{-r_i}$ are relatively prime. Hence there are integers $s_1,s_2,\ldots, s_n$ such that $s_1m_1+s_2m_2+\cdots+s_nm_n=1$. Thus $x=x^{s_1m_1}\cdot x^{s_2m_2}\cdot\ldots\cdot x^{s_nm_n}$, where $x^{m_i}\in G_{p_i}$.

(2) Obviously, $G'\geq \bigoplus_{p\in \mathbb P} G_p'$. Conversely, let $[x,y]\in G'$, where $x=(x_p)_{p\in \mathbb P}, y=(y_p)_{p\in \mathbb P}\in G$. Then we have $[x,y]=([x_p,y_p])_{p\in \mathbb P}\in \bigoplus_{p\in \mathbb P} G_p'$.

(3) We define a homomorphism $\psi:\bigoplus_{p\in \mathbb P} G_p\to \bigoplus_{p\in \mathbb P}(G_p/G_p')$ by $\psi((x_p)_{p\in \mathbb P})= (x_pG_p')_{p\in \mathbb P}$. Then $\psi$ is  surjective, and $\ker\psi= \bigoplus_{p\in \mathbb P}G_p'$. Hence, $\bigoplus_{p\in \mathbb P}(G_p/G_p')\cong \bigoplus_{p\in \mathbb P}G_p/\bigoplus_{p\in \mathbb P}G_p'\cong G/G'$ by (1) and (2).
\end{proof}

\subsection{The general case}\label{gc}

Recall that if $A$ and $B$ are subgroups of a group $G$, then $AB = BA$ if and only if $AB$ is a subgroup of $G$. In the context of topological groups, a similar equivalence holds: if $H$ and $L$ are subgroups, then $\overline{HL}=\overline{LH}$ if and only if $\overline{HL}$ is a subgroup of $G$. In the next definition we consider the closed subgroups satisfying this property.

\begin{definition}\cite{K}\label{def:tqh}
Let $G$ be a topological group.
\ben
\item A closed subgroup $H$ is called topologically quasinormal (tqn for short) if $\overline{HL}=\overline{LH}$ for every closed subgroup $L$ in $G$.
\item $G$ is called topologically quasihamiltonian (tqh for short) if every closed subgroup is topologically quasinormal.
\een
\end{definition}

Note that the topological groups satisfying Definition \ref{def:tqh}(2) with the discrete topology are exactly the quasihamiltonian groups considered in \S\ref{QuasiHamilt}.

\begin{remark} \label{remark:closed}
Let $X$ and $Y$ be closed subgroups of a topological group $G$. It is known that if $X$ is compact,
then $XY$ is closed in $G$. So, if $G$ is tqh, then $XY$ is a closed subgroup of $G$.   So, if $G$ is tqh, then $XY$ is a closed subgroup of $G$.  In case also $Y$ is compact then the subgroup $XY$ is compact too. In particular,  one can deduce the following.

\ben \item  Every compact subgroup $X$ of a tqh group commutes with any closed subgroup $Y$. In particular, if $x,y\in G$ are torsion elements, then
the cyclic subgroups $X$ and $Y$ they generate commute, so $XY$ is a finite subgroup. This proves that:

\ben [(a)] \item the torsion part $t(G)$ of $G$ is a locally finite subgroup of $G$;

\item the union $comp(G)$ of all compact subgroups of $G$ is a (compactly covered) subgroup of $G$ containing the subgroup $t(G)$.
\een

\item If  $\phi \in \End (G)$,  $U$ is a compact open subgroup of $G$, and $n\in\N_+$ then  $T_n(\phi, U)$ is a compact open subgroup of $G$.
\een

\end{remark}

From what we said above, it follows that in a locally compact tqh group $G$ the subgroup $comp(G)$ is also the maximum compactly covered subgroup of $G$.
It is worth noting that Herfort, Hofmann and Russo  gave in  \cite[Theorem 8.4]{HHR} a structure theorem for  compactly covered, totally disconnected,  locally compact  tqh groups.

It is clear from the definition that the class of tqh groups is stable under taking closed subgroups. Indeed, it is stable under taking arbitrary subgroups, as the first item in the following fact, due to K\"{u}mmich, shows.

\begin{fact}\ \label{lem:s2q}\cite{K}
\ben
\item The class of tqh groups is stable under taking  subgroups and Hausdorff quotients.
\item Let $G$ be a locally compact tqh group. If $G/c(G)$ is compact, then $G$ is either abelian or totally disconnected.
\een
\end{fact}

 For a tqh group $G$, both $t(G)$ and $comp(G)$ are fully invariant subgroups of $G$, but they may fail to be closed. In case $G$ is also locally compact and totally disconnected, then $comp(G)$ is an open normal subgroup of $G$. So, the quotient group $G/comp(G)$ is a torsion-free quasihamiltonian group by Fact \ref{lem:s2q}(1). It follows that $G/comp(G)$ is abelian by Fact \ref{quasi}(3).

\begin{fact}\cite[Lemma 2.15]{BWY}\label{fac:conn}
Let $G$ be a connected locally compact group. If $G$ is compactly covered, then it is compact. In particular, if $H$ is a compactly covered locally compact group, then its connected component $c(H)$ is compact.
\end{fact}

\begin{corollary}\label{cor:abtd}
Let $G$ be a compactly covered locally compact tqh group. Then $G$ is either abelian or totally disconnected.
\end{corollary}
\begin{proof}
Let $G$ be a compactly covered locally compact tqh group and assume in addition that $c(G)$ is non-trivial. We have to prove that $G$ is abelian. To this aim,
fix $a,b\in G$ and let us see that $a$ and $b$ commute. Since $G$ is compactly covered, the subgroups $\overline{\langle a\rangle}$ and $\overline{\langle b\rangle}$ are compact.
As $G$ is tqh, we deduce that $\overline{{\langle a,b\rangle}}=\overline{\langle a\rangle}\cdot \overline{\langle b\rangle}$.  By Fact \ref{fac:conn}, $c(G)$ is compact so Fact \ref{lem:s2q}(1) implies that $\overline{{\langle a,b\rangle}}c(G)$ is a compact tqh group that is not totally disconnected. Using
Fact \ref{lem:s2q}(2)  we conclude that $\overline{{\langle a,b\rangle}}c(G)$ is abelian. In particular, $a$ and $b$ commute.
\end{proof}

Mukhin \cite{M} proved the following stronger result, which asserts that one can remove the ``compactly covered'' assumption in Corollary \ref{cor:abtd}.

\begin{fact}\label{Mukhin}
If $G$ is a locally compact tqh group, then $G$ is either abelian or totally disconnected.
\end{fact}

\begin{lemma}\label{lem:copsub}
Every element of a compactly covered, locally compact tqh group $G$ is contained in a compact open subgroup.
\end{lemma}
\begin{proof}
By Fact \ref{Mukhin}, $G$ is either abelian or totally disconnected.  In case $G$ is abelian the thesis follows from \cite[Corollary 2.1]{GBST}. Now assume that $G$ is totally disconnected and fix $x\in G.$  As $G$ is totally disconnected it has a local base at identity consisting of compact open subgroups. Pick a compact open subgroup $U$ and observe that $\overline{\langle x\rangle}$ is compact since $G$ is compactly covered. It follows that $U\overline{\langle x\rangle}$ is a compact open subgroup of $G$ containing $x$, as $G$ is a tqh group.
\end{proof}

Clearly, a topological abelian group is tqh.  In \cite[Proposition 2.2]{DGB}, Giordano Bruno and Dikranjan proved that if $G$ is a compactly covered locally compact abelian group, then $\mathcal B(G)$ is cofinal in $\mathcal C(G)$. We extend this result as follows.

\begin{proposition}\label{prop:cof}
Let $G$ be a compactly covered locally compact tqh  group. Then $\mathcal B(G)$ is cofinal in $\mathcal K(G)$. In particular, $\mathcal B(G)$ is cofinal in $\mathcal C(G)$.
\end{proposition}
\begin{proof}
Let $K$ be a compact subset of $G$. By Lemma \ref{lem:copsub}, for every $k\in K$ there exists $A_k\in \mathcal B(G)$ such that $k \in A_k$, so $K\subseteq \bigcup_{k\in K}A_k$. Using the compactness of $K$ we find finitely many $k_1, \ldots, k_n \in K$ such that $K\subseteq \bigcup_{i=1}^n A_{k_i}$. As $G$ is topologically quasihamiltonian, we have that $U=A_{k_1}\cdots A_{k_n}\in \mathcal B(G)$ satisfies
$K\subseteq U$.
\end{proof}

If $G$ is a  compactly covered locally compact tqh group, and $H$ is a closed %(respectively, closed and normal)
subgroup of $G$, Fact \ref{lem:s2q}(1) yields that we can apply Proposition \ref{prop:cof} to $H$, so that $\mathcal B(H)$ is cofinal in $\mathcal C(H)$. Similarly, if $H$ is also normal in $G$, then $\mathcal B(G/H)$ is cofinal in $\mathcal C(G/H)$.
The following lemma presents smaller cofinal subfamilies of $\mathcal C(H)$ and $\mathcal C(G/H)$ in the same setting.

\begin{lemma}\label{cofinal:families}
Let $G$ be a  compactly covered locally compact tqh group, and let $H$ be a closed subgroup of $G$.  Then:
\ben
\item  the family $\mathcal B_G(H)=\{U \cap H : U \in \mathcal B(G) \}$ is cofinal in $\mathcal C(H)$;
\item  if $H$ is also normal in $G$, then the family $\mathcal B_G(G/H)=\{\pi U : U \in \mathcal B(G) \}$ is cofinal in $\mathcal C(G/H)$, where $\pi: G \to G/H$ is the canonical projection.
\een
\end{lemma}
\begin{proof}
 By the above discussion, it suffices to prove that $\mathcal B_G(H)$ is cofinal in $\mathcal B(H)$,  and that $\mathcal B_G(G/H)$ is cofinal in $\mathcal B(G/H)$ when $H$ is also normal in $G$.
	
(1) Let $V \in \mathcal B(H)$. In particular, $V$ is a compact subgroup of $G$. By  Proposition \ref{prop:cof}, there exists $U\in \mathcal B(G)$ containing $V$. Thus, $U \cap H \geq  V$.
	
(2) Let $V \in \mathcal B(G/H)$. For every $v \in V$, let $u \in G $ with $\pi(u)=v$. By Lemma \ref{lem:copsub}, there exists $K_u\in \mathcal B(G)$ such that $u\in K_u$.
By the compactness of $V$, there exist $u_1, \ldots, u_n \in G$ such that $V \subseteq \bigcup_{i=1}^n \pi (K_{u_i})$. As $G$ is tqh, we deduce that $U =K_{u_1}K_{u_2}\cdots K_{u_n}\in\mathcal B(G)$, and $V\leq\pi (U)$.
\end{proof}

\section{Algebraic Entropy}
 Let $G$ be a group, $U\leq G$, and $T\subseteq G$  a disjoint union of right cosets of $U$.
The so-called {\em generalized right index of $U$ in $T$} is the number of those cosets, and we denote it by $[T:U]$. Obviously, when also $T$ is a subgroup of $G$, the generalized right index of $U$ in $T$ coincides with the usual index.

\begin{proposition}\label{prop:conv}
Let $G$ be a locally compact group and $\phi \in \End(G)$.
\ben
\item If $U \in \mathcal B(G)$, then $H_{alg} (\phi, U) = \limsup_{n \to \infty} \frac{\log [T_n(\phi,U) :U] }{n}$, where $[T_n(\phi,U) :U]$ is the generalized right index of $U$ in $T_n(\phi,U)$.
\item  If in addition $\phi^n(U)U=U\phi^n(U)$ for every $n\in\N$, then $T_n(\phi, U)\leq G$ so in this case the generalized index above is the usual index, and moreover \[H_{alg} (\phi, U) = \lim_{n \to \infty} \frac{\log [T_n(\phi,U) :U] }{n}.\]
\een
\end{proposition}
\begin{proof}
 $(1)$ If $U \in \mathcal B(G)$, then $T_n = T_n(\phi, U)$ is a compact subset of $G$, and it is a disjoint union of right cosets of $U$. In particular the generalized right index $[T_n:U]$ is finite, and using the properties of $\mu$, we obtain $\mu (T_n) = [T_n:U] \mu (U)$, so that $\log \mu (T_n) = \log [T_n:U] + \log \mu (U)$. As $\log \mu (U)$ does not depend on $n$, passing to the limit superior for $n \to \infty$ we obtain
\begin{equation}\label{eq:fir}
H_{alg} (\phi, U) = \limsup_{n \to \infty} \frac{\log [T_n(\phi,U) :U] }{n}.
\end{equation}

 $(2)$ For $U \in \mathcal B(G)$ satisfying $\phi^n(U)U=U\phi^n(U)$ for every $n\in\N$, one can prove by induction that $\phi^n(U) T_n = T_n \phi^n(U)$ for every $n\in\N$ (the interested reader can find a proof of this in \cite[Lemma 3.1]{GBST}).
For $n\in\N$, we conclude that $T_{n+1}=T_n \phi^n(U)$ is a subgroup of $G$. Since $T_n$ is also compact and open, we have $T_n\in \mathcal B(G)$.

Let $t_n =[T_n:U]$. Then $t_n$ divides $t_{n+1}$, as $U \leq T_n \leq T_{n+1}$, and let $\beta_n = t_{n+1}/t_n = [T_{n+1}:T_n]$. Now we show that the sequence of integers $\{ \beta_n \}_{n \geq 1}$ is weakly decreasing. Indeed,
\[\beta_n = [T_{n+1}:T_n] \geq [\phi(T_{n+1}): \phi (T_n)] \geq [U\phi(T_{n+1}): U\phi (T_n)] = [T_{n+2}: T_{n+1}] = \beta_{n+1}.\]
In particular $\{ \beta_n \}_{n \geq 1}$ stabilizes, so let $n_0 \in \N$, and $\beta \in \N$ be such that for every $n \geq n_0$ we have $\beta_n = \beta$, i.e., $t_n = \beta^{n-n_0} t_{n_0}$.  Then item (1) gives
\begin{equation}\label{eq:sec}
H_{alg} (\phi, U) = \limsup_{n \to \infty} \frac{\log \beta^{n-n_0} t_{n_0} }{n} = \log \beta =
\lim_{n \to \infty} \frac{\log t_n }{n}. \qedhere
\end{equation}
\end{proof}

Note that under the assumptions of Proposition \ref{prop:conv}(2), $H_{alg} (\phi, U)=0$ if and only if
$\beta=1$
in Equation (\ref{eq:sec}).
This happens exactly when $T_{n_0+1}(\phi, U)=T_{n_0}(\phi, U)$ for
some $n_0$ big enough (and then the equality $T_{n+1}(\phi, U)=T_{n}(\phi, U)$ holds for every $n \geq n_0$).

\begin{corollary}\label{cor:tqngen}
If $G$ is a
 locally compact
 group and $\phi \in \End(G)$, then
\begin{equation*}
H_{alg} (\phi, U) = \lim_{n \to \infty} \frac{\log [T_n(\phi,U) :U] }{n},
\end{equation*}
for every tqn  $U \in \mathcal B(G)$.
%In particular,

 If $G$ is also compactly covered and tqh,  then
\begin{equation}\label{eq:alg}
h_{alg}(\phi) = \sup \{ H_{alg} (\phi, U) : U \in \mathcal B(G) \}.	
\end{equation}
\end{corollary}
\begin{proof}
Fix $U \in \mathcal B(G)$ that is also tqn and observe that $\phi^n(U)U=U\phi^n(U)$ for every $n\in\N_+$.  By Proposition \ref{prop:conv}(2), we have $H_{alg} (\phi, U) = \lim_{n \to \infty} \frac{\log [T_n(\phi,U) :U] }{n}$.

Now assume that $G$ is  also compactly covered and tqh.  By Proposition \ref{prop:cof},  $\mathcal B(G)$ is cofinal in $\mathcal C(G)$ and $h_{alg}(\phi) = \sup \{ H_{alg} (\phi, U) : U \in \mathcal B(G) \}.$
\end{proof}

\begin{remark}\label{p-group}
Let $G$ be a \cclctqh\ $p$-group, and $\phi \in \End(G)$.  By Equation \ref{eq:alg}, $h_{alg}(\phi)>0$ if and only if there is $U \in \mathcal B(G)$ such that $\beta_n= [T_{n+1}(\phi, U):T_n(\phi, U)] \neq 1$ for every $n$. In this case, $\beta_n$ is a positive power of $p$, so $h_{alg}(\phi)\geq \log p$.
\end{remark}

 Lemma \ref{cofinal:families} and Corollary \ref{cor:tqngen} immediately give the following.
\begin{corollary}\label{cor:halg:induced:maps}
Let $G$ be a  group, $\phi \in \End(G)$, and $H$ be a closed $\phi$-invariant subgroup of $G$. Then:
\ben
\item $\phi \restriction_H \in \End(H)$, and $$h_{alg} (\phi \restriction_H) = \sup \{ H_{alg} (\phi, U) : U \in \mathcal B_G(H) \};$$
\item if $H$ is also normal, and $\bar \phi : G/H \to G/H$ denotes the induced map, then $\bar \phi \in \End(G/H)$, and
	$$h_{alg} (\bar \phi) = \sup \{ H_{alg} (\phi, U) : U \in \mathcal B_G(G/H) \}.$$
\een
\end{corollary}

\subsection{Algebraic entropy on torsion Quasihamiltonian groups}
Let $G$ be a locally finite (discrete) group and $\phi\in \End(G)$. It is proved in  \cite{DG-islam}(see Remark 5.1.5(b)  and the paragraph  before Example 5.6.1) that for every finite subgroup $F$ of $G$, we have  $H_{alg} (\phi, F) = \lim_{n \to \infty} \frac{\log |T_n(\phi,F)| }{n}$, and
\begin{equation}\label{eq:hdis}
h_{alg}(\phi) = \sup \{ H_{alg} (\phi, F) : F\ \text{is a finite subgroup of} \  G\}.
\end{equation}	

Moreover, if $T_n(\phi, F)$ is a subgroup of $G$ (e.g., $G$ is torsion quasihamiltonian), then we also have $H_{alg} (\phi, F)=  \lim_{n \to \infty} \frac{\log [T_n(\phi, F): F] }{n},$ as $|T_n(\phi,F)|=[T_n(\phi, F): F]\cdot|F|$.

\begin{lemma}\label{lem:hate}
Let $\phi:G\to G$ be an endomorphism of a torsion quasihamiltonian group, and let $F$ be a finite subgroup of $G$. If $h_{alg} (\phi)=0$, then there exists $m>0$ such that $T_m(\phi, F)$ is $\phi$-invariant.
\end{lemma}
\begin{proof}
Note that as $G$ is torsion quasihamiltonian, the assumptions of Proposition \ref{prop:conv}(2) are satisfied. Since $h_{alg} (\phi)=0$, we have $H_{alg} (\phi,F)=0$. So, Equation (\ref{eq:sec}) implies that there exists $m>0$ such that $T_{m}(\phi,F)= T_n(\phi,F)$ when $n\geq m$. It follows that $T_m(\phi, F)$ is $\phi$-invariant.
\end{proof}

The following lemma extends \cite[Lemma 1.5]{DGSZ}, where the torsion abelian case was considered.

\begin{lemma}\label{lem:zero} Let $\phi:G\to G$ be an endomorphism of a torsion quasihamiltonian group, $H$ a $\phi$-invariant normal subgroup of $G$, and $\bar \phi: G/H \to G/H$ the induced endomorphism. If $h_{alg} (\bar \phi)=0$, then $h_{alg} (\phi) =h_{alg} (\phi \upharpoonright_H)$.
\end{lemma}
\begin{proof}
The definition of entropy easily yields that $h_{alg}(\phi)\geq h_{alg}(\phi\upharpoonright_H)$. Then it suffices to show that $h_{alg}(\phi)\leq h_{alg}(\phi\upharpoonright_H)$,  i.e. that $H_{alg}(\phi, F)\leq h_{alg}(\phi\upharpoonright_H)$ for an arbitrary finite subgroup $F$ of $G$.

Let $\pi : G\to G/H$ be the canonical homomorphism and let $F_1=\pi(F)$. Since $\bar\phi$ has zero entropy, there exists $m > 0$ such that the subgroup $T_m(\bar\phi, F_1)$ is $\bar\phi$-invariant by Lemma \ref{lem:hate}. It follows that $\phi ^{m}(F) \leq T_m(\phi, F)\cdot H$. As $F$ is finite and $H$ is locally finite by Remark \ref{rem:per}(3), there exists a finite subgroup $F_2$ of $H$ such that $\phi ^{m}(F)\leq T_m(\phi, F)\cdot F_2$. This implies that
\[
\phi ( T_m(\phi, F) )\leq T_m(\phi, F)\cdot F_2,
\]
from which, by induction on $k>0$, we get 	
\begin{equation}\label{eq:tmtk}
\phi^{k}( T_m(\phi, F) )\leq T_m(\phi, F)\cdot T_k(\phi,F_2).
\end{equation}

Moreover, by induction on $k$ it can be proved that
\begin{equation}\label{eq:sectmtk}
\phi^{m} ( T_k(\phi, F) ) \leq \phi^{k} ( T_m(\phi, F) ), \ \forall k\geq m.
\end{equation}

Using Equations (\ref{eq:tmtk}) and  (\ref{eq:sectmtk}) we obtain
\begin{equation}\label{eq:third}
\phi^{m} ( T_k(\phi, F) ) \leq T_m(\phi, F)\cdot T_k(\phi,F_2), \ \forall k\geq m.
\end{equation}
Now let $n = m + k$ for some $k \geq m$.  Then by Equation (\ref{eq:third})
\[
T_n(\phi, F)= T_m(\phi, F)\phi^{k} ( T_m(\phi, F) ) \leq T_m(\phi, F)\cdot T_k(\phi, F_2),
\]
so we have
\[
\log|T_n(\phi, F)| \leq \log|T_m(\phi, F)|+ \log|T_k(\phi, F_2)|.
\]

Since $m$ is fixed, dividing by $n$, and letting $n\to \infty$ (so $k\to \infty$ as well), we deduce that
\[
H_{alg}(\phi, F)\leq H_{alg}(\phi, F_2)\leq h_{alg}(\phi\upharpoonright_H). \qedhere
\]
%as needed.
\end{proof}

In the following result we prove the Addition Theorem for a torsion quasihamiltonian group $G$ and a subgroup $H$, in the particular cases when either $[G:H]$ is finite, or $H$ is finite.

\begin{lemma}\label{lem:finite}
Let $\phi:G\to G$ be an endomorphism of a torsion quasihamiltonian group, and $H$ be a $\phi$-invariant normal subgroup of $G$.
\ben
\item If $[G:H]$ is finite, then  $h_{alg} (\phi) =h_{alg} (\phi \upharpoonright_H)$.
\item If $H$ is finite, then $h_{alg} (\phi) =h_{alg} (\bar \phi)$.
\een
\end{lemma}

\begin{proof}
(1) This is a consequence of Lemma \ref{lem:zero}   as  the endomorphisms of a finite group have zero algebraic entropy.

(2) By  \cite[Lemma 5.1.6(b)]{DG-islam}, we have $h_{alg} (\phi) \geq h_{alg} (\bar \phi)$. So, it suffices to show that  $h_{alg} (\phi) \leq h_{alg} (\bar \phi)$.
Let $|H|=m\in \N_+$ and fix a finite subgroup $F$ of $G$ and $n\in \N_+$. Since $T_n(\phi,F)/H\cap T_n(\phi,F)$ is isomorphic to $\pi(T_n(\phi,F))$ we obtain
\[
|T_n(\phi,F)|=|H\cap T_n(\phi,F)|\cdot|\pi(T_n(\phi,F))|\leq m\cdot |\pi(T_n(\phi,F))|.
\]
As $\pi(T_n(\phi,F))=T_n(\bar\phi,\pi (F))$  we deduce that $|T_n(\phi,F)|\leq m\cdot |T_n(\bar\phi,\pi (F))|$, which implies that $H_{alg}(\phi, F)\leq H_{alg}(\bar\phi, \pi (F))\leq h_{alg} (\bar \phi)$. Taking the supremum over $F$, this proves $h_{alg} (\phi) \leq h_{alg} (\bar \phi)$.
\end{proof}

\section{Basic properties}\label{basic-sec}

In this section we give the basic properties of $h_{alg}$ of endomorphisms of \cclctqh\ groups. We start by showing that the identity map of such groups has zero algebraic entropy.

\begin{lemma}\label{h:alg:id}
If $G$ is \cclctqh, then $h_{alg}(id_G)=0$.
\end{lemma}
\begin{proof}
If $G$ is a locally compact group, then $H_{alg}(id_G, U)=0$ for every subgroup $U\in\mathcal C(G)$
by \cite[Lemma 4.1]{GBST}.
So, if $G$ is also compactly covered tqh group, then $h_{alg}(id_G)=0$ by Equation \eqref{eq:alg}.
\end{proof}

Invariance under conjugation was proved in general for endomorphisms of locally compact abelian groups  by Virili \cite[Proposition 2.7(1)]{V}. This property holds true without the abelian assumption as noted in \cite{DG-islam} (for a proof, see \cite[Corollary 4.3]{GBST}).
\begin{lemma}[Invariance under conjugation]\label{inv:conj}
Let $\alpha:G\to G_1$ be  a topological isomorphism of locally compact groups. If $\phi\in \End(G)$ and  $\psi=\alpha \phi\alpha^{-1}$,  then  $H_{alg}(\phi ,K) = H_{alg}(\psi, \alpha (K))$  for every $K\in \mathcal C(G)$. In particular, $h_{alg}(\phi)=h_{alg}(\psi).$
\end{lemma}

Now we prove the Logarithmic Law for the algebraic entropy, with respect to endomorphisms of \cclctqh\ groups.

\begin{lemma}[Logarithmic Law]
Let $G$ be a \cclctqh\ group and $\phi\in \End(G).$ Then  $h_{alg}(\phi^m)=m\cdot h_{alg}(\phi)$ for every $m\in \N.$
\end{lemma}
\begin{proof}
Since the assertion is clear when $m=0$ (see Lemma \ref{h:alg:id}), we may fix $m>0$. Let us show first that $h_{alg}(\phi^m)\leq m\cdot h_{alg}(\phi)$. Let $n\in \N_+$ and $U\in \mathcal B(G)$. As $U$ and $\phi^n(U)$ are compact subgroups in the tqh group $G$, also $T_n(\phi^m,U)$ and $T_{mn-m+1}(\phi, U)$ are subgroups of $G$ by Proposition \ref{prop:conv}(2). Moreover, we have $U\leq T_n(\phi^m,U)\leq T_{mn-m+1}(\phi, U)$. This implies that
\begin{equation*}\begin{split}
H_{alg}(\phi^m, U)&=\lim_{n\to \infty}\frac{\log[T_n(\phi^m,U):U]}{n}\leq
\\ &\leq \lim_{n\to \infty} \frac{\log[T_{mn-m+1}(\phi, U):U]}{mn-m+1}\cdot \lim_{n\to \infty}\frac{mn-m+1}{n} = mH_{alg}(\phi, U),
\end{split}\end{equation*}
and taking the suprema over $U\in \mathcal B(G)$ we obtain $h_{alg}(\phi^m)\leq m\cdot h_{alg}(\phi)$.

To prove the converse inequality, let $U\in \mathcal B(G)$, and $n \in \N_+$. If $V=T_m(\phi,U)$, Proposition \ref{prop:conv}(2) yields that $V\in \mathcal B(G)$, as well as
$$W = T_n(\phi^m, V)=T_{nm}(\phi,U)\in \mathcal B(G).$$

Obviously $U \leq V \leq W$, so $\log[W:V] = \log[W:U]-\log[V:U]$. Then
\begin{equation*}\begin{split}
h_{alg}(\phi^m)\geq H_ {alg} (\phi^m, V) &= \lim_{n\to \infty}\frac{\log [T_n ( \phi^m, V ):V]}{n} = \\
&= m\lim_{n\to \infty}\frac{\log [W:U]}{mn} - \frac{\log [V:U]}{mn}= m H_ {alg} (\phi,U).
\end{split}\end{equation*}
As the above inequality holds for every $U\in \mathcal B(G)$, this proves that $h_{alg}(\phi^m)\geq m\cdot h_{alg}(\phi)$, and thus $h_{alg}(\phi^m)= m\cdot h_{alg}(\phi)$.
\end{proof}

The next  property immediately follows from the stronger result proved in Proposition \ref{first:half:add:thm}.

\begin{proposition}[Monotonicity]\label{monotonicity}
Let $G$ be a \cclctqh\ group, $\phi \in \End(G)$, $H$ a closed $\phi$-invariant subgroup of $G$.
\ben
\item Then $h_{alg} (\phi) \geq h_{alg} (\phi \restriction_H)$.
\item If $H$ is normal and  $\bar \phi : G/H \to G/H$ the induced map, then $h_{alg} (\phi) \geq  h_{alg} (\bar \phi)$.
\een
\end{proposition}

\subsection{Topologically quasinormal subgroups of locally compact groups}\label{sec:LF}

Let $G$ be a compactly covered locally compact tqh group. It was shown in \cite[Example 4.7]{GBST} that if $\phi\in \Aut(G)$ then $h_{alg}(\phi)$ and $h_{alg}(\phi^{-1})$ may be different.

Next we give the precise relation between $h_{alg}(\phi)$ and $h_{alg}(\phi^{-1})$, which extends \cite[Proposition 2.7(3)]{V} (see also \cite[Proposition 4.9]{GBST}).

\smallskip
For a locally compact group $G$, let $\Aut(G)$ denote the group of topological automorphisms of $G$. If $\mu$ is a left Haar measure on $G$, the \emph{modulus} is a group homomorphism $\Delta_G:\Aut(G)\to\R_+$ such that $\mu(\phi E)=\Delta_G(\phi)\mu(E)$ for every $\phi\in \Aut(G)$ and every measurable subset $E$ of $G$ (see \cite{HR}). If $G$ is either compact or discrete, then $\Delta_G(\phi)=1$ for every $\phi\in \Aut(G)$. We also denote $\Delta_G$ simply by $\Delta$.

\begin{lemma}\label{Delta}
Let $G$ be a locally compact group, $\phi\in \Aut(G)$, and $U\in\mathcal B(G)$ be a tqn subgroup. Then
\ben
\item  $\Delta(\phi)=\frac{\mu(\phi(U))}{\mu(U)}=\frac{[U\phi(U):U]}{[U\phi(U):\phi(U)]}$.
\item If $V$ is a compact subgroup of $G$ and $V\supseteq U\phi(U)$, then $[V:U]=[V:\phi(U)]\cdot\Delta(\phi)$.
\een
\end{lemma}

The above lemma extends \cite[Lemma 4.8]{GBST} in which the case $U\in\mathcal N(G)$ is considered. The same arguments used there apply to a tqn $U\in\mathcal B(G)$.

\begin{proposition}\label{h:alg:phi:inverse}
Let $G$ be a locally compact group, $\phi\in \Aut(G)$, and $U\in\mathcal B(G)$ be a tqn subgroup. Then
\[H_{alg}(\phi^{-1},U)= H_{alg}(\phi,U)-\log\Delta(\phi).\]

In particular, if $G$ is a \cclctqh\ group, then
\[h_{alg}(\phi^{-1})=h_{alg}(\phi)-\log\Delta(\phi).\]
\end{proposition}
\begin{proof}
Since $U$ is a compact open tqn subgroup of $G$, both $T_n(\phi,U)$ and $T_n(\phi^{-1},U)$ are subgroups of $G$
by Proposition \ref{prop:conv}(2).
For every $n\in\N_+$ let $t_n=[T_n(\phi,U):U]$ and $t_n^*=[T_n(\phi^{-1},U):U]$.
Moreover, $H_{alg}(\phi,U)=\log\beta$ and $H_{alg}(\phi^{-1},U)=\log \beta^*$, where $\beta$ and $\beta^*$ are respectively the values at which the sequences $\beta_n=\frac{t_{n+1}}{t_{n}}$ and $\beta^*_n=\frac{t_{n+1}^*}{t_{n}^*}$ stabilize (see Equation \eqref{eq:sec}).

Let $n\in\N_+$. Since $G$ is tqh, we have that
\[\phi^{n-1}(T_n(\phi^{-1},U))=T_n(\phi,U)\]
is a compact subgroup of $G$. Since $\phi^{n-1}$ is an automorphism, Lemma \ref{Delta}(2) gives
\begin{equation*}\begin{split}
t_n^* &=[T_n(\phi^{-1},U):U]=[\phi^{n-1}(T_n(\phi^{-1},U)):\phi^{n-1}(U)]=\\&=[T_n(\phi,U):\phi^{n-1}(U)]=[T_n(\phi,U):U]\cdot\frac{1}{\Delta(\phi^{n-1})}=t_n\cdot\frac{1}{\Delta(\phi^{n-1})}.
\end{split}\end{equation*}
Therefore, since $\Delta$ is a homomorphism, for every sufficiently large $n\in\N_+$, $$\beta=\frac{t_{n+1}}{t_n}=\frac{t_{n+1}^*}{t_n^*}\cdot\frac{\Delta(\phi^{n})}{\Delta(\phi^{n-1})}=\beta^*\cdot\Delta(\phi).$$ Then, $\log\beta=\log\beta^*+\log\Delta(\phi)$, that is, the first assertion of the proposition.

The second part of the statement follows from the first one, taking the supremum for $U\in\mathcal B(G)$ in view of Equation \eqref{eq:alg}.
\end{proof}

\bigskip

In what follows we obtain the so-called Limit-free Formula for the algebraic entropy. We start introducing the following useful subgroups.

\begin{definition}\label{def:umin}
Let $G$ be a locally compact group, $\phi\in \End(G)$ and $U$ be a compact open tqn subgroup of $G$. Define:
\ben \item $U^{(0)}=U;$
\item $U^{(n+1)}=U\phi^{-1}U^{(n)}$ for every $n\in \N;$
\item $U^-=\bigcup_{n\in\N} U^{(n)}.$
\een
\end{definition}

Since $U$ is a compact open tqn subgroup of $G$, one can prove by induction that $U^{(n+1)} = \phi^{-1}U^{(n)} U$ is an open subgroup of $G$ such that $U^{(n)} \leq U^{(n+1)}$ for every $n \in \N$. Moreover, for every $U\in \mathcal B(G)$, also the subgroup $U^-$ is open in $G$.
The following lemma, which collects some of the properties of   $U^-$,  generalizes \cite[Lemma 5.2]{GBST}.
 Its proof is similar.

\begin{lemma}\label{lem:pro}
Let $G$ be a locally compact group, $\phi\in \End(G)$ and $U$ be a compact open tqn subgroup of $G$. Then:
\ben
\item $\phi^{-1}U^{-} \leq U^-$;
\item if $H\leq G$ is such that $U\leq H$ and $\phi^{-1}H\leq H,$ then $U^{-}\leq H$;
\item $U^- = U\phi^{-1}U^-$;
\item the index $[U^- :\phi^{-1}U^-] = [U:U\cap \phi^{-1}U^-]$ is finite;
\item for every $n\in \N^{+}, \phi^{-n}T_n = \phi^{-1}U^{(n-1)}. $
\een
\end{lemma}

Replacing $U\in \mathcal N(G)$ in \cite[Proposition 5.3]{GBST} with $U\in \mathcal B(G)$ that is also tqn we obtain the following proposition. One can refer to the proof there.

\begin{proposition}[Limit-free Formula] \label{prop:limit-free}
Let $G$ be a locally compact group. If $\phi\in \End(G)$ and $U$ is a compact open tqn subgroup of $G$, then $$H_{alg} (\phi,U)= \log[U^- :\phi^{-1}U^-].$$
\end{proposition}

The next proposition is a consequence of the above Limit-free Formula, and will be used in the proof of the Addition Theorem \ref{addthm}.

\begin{proposition}\label{LFcoralg}
Let $G$ be a compactly covered locally compact tqh group, and $\phi\in \End(G)$. Then
$$h_{alg}(\phi)=\sup\{\log[A:\phi^{-1}A]: A\leq G,\ A\ \text{open},\  \phi^{-1}A\leq A,\ [A:\phi^{-1}A]<\infty\}=:s.$$
\end{proposition}
\begin{proof}
Using Proposition \ref{prop:limit-free} we obtain
\[h_{alg}(\phi)=\sup\{H_{alg}(\phi, U): U\in \mathcal B(G)\}=\sup\{\log[U^-:\phi^{-1}U^-]: U\in \mathcal B(G)\}.\]
Then  $h_{alg}(\phi)\leq s$, as for every $U\in \mathcal B(G)$, $U^-$ is an open subgroup of $G$ such that $\phi^{-1}U^-\leq U^-$ and $[U^-:\phi^{-1}U^-]<\infty$ by Lemma \ref{lem:pro}(4).

For the converse inequality, fix an open normal subgroup $A$ of $G$ such that $\phi^{-1}A\leq A$ and $[A:\phi^{-1}A]<\infty.$ Our aim is to find $U\in \mathcal B(G)$ such that $[U^-:\phi^{-1}U^-]\geq [A:\phi^{-1}A]$. Since $[A:\phi^{-1}A]<\infty$, there exists a finitely generated subgroup $F\leq A$ such that $A=\phi^{-1}(A)F$. By Proposition \ref{prop:cof}, there exists a compact subgroup $K\leq G$ such that $F\leq K.$ We claim that
\begin{equation} \label{eq:coadd}
A=\phi^{-1}(A)U,
\end{equation}
where $U=K\cap A\in \mathcal B(G).$ First observe that
$$F\leq A\cap K = U\leq A.$$
Taking also into account that $\phi^{-1}A\leq A$ we obtain
\[A=\phi^{-1}(A)F\leq \phi^{-1}(A)U,\] which proves Equation \eqref{eq:coadd}.

We now show that $[U^-:\phi^{-1}U^-]\geq [A:\phi^{-1}A]$.
Since $U\leq A$ and $\phi^{-1}A\leq A$, Lemma \ref{lem:pro}(2) gives $U^-\leq A$, so  $\phi^{-1}U^-\leq \phi^{-1}A\leq A$.
Now by Equation \eqref{eq:coadd},
$$[U\phi^{-1}U^-:\phi^{-1}U^-]\geq [(U\phi^{-1}(U^-))\phi^{-1}A:\phi^{-1}(U^-)\phi^{-1}A]=[A:\phi^{-1}A].$$
Finally, $[U^-:\phi^{-1}U^-]=[U\phi^{-1}U^-:\phi^{-1}U^-]$ by Lemma \ref{lem:pro}(3).
\end{proof}

\section{Addition Theorems}\label{sec:AT}
In this section, we adopt the notation AT$(G, H, \phi)$ for a group $G$, $\phi \in \End(G)$ and a closed normal $\phi$-invariant subgroup $H$ of $G$, to indicate briefly that ``$h_{alg} (\phi) =h_{alg} (\phi \upharpoonright_H)+h_{alg} (\bar\phi) $ holds for the triple $(G, H, \phi)$'', where $\bar \phi : G/H \to G/H$ is the induced map.

We begin this section proving one inequality of AT$(G, H, \phi)$ in the class of compactly covered locally compact tqh groups.

\begin{proposition}\label{first:half:add:thm}
Let $G$ be a compactly covered locally compact tqh group, $\phi \in \End(G)$, $H$ a closed normal $\phi$-invariant subgroup of $G$, and $\bar \phi : G/H \to G/H$ the induced map. Then
\begin{equation*}
h_{alg} (\phi) \geq h_{alg} (\bar \phi) + h_{alg} (\phi \restriction_H).
\end{equation*}
\end{proposition}
\begin{proof}
For $U\in \mathcal B(G)$ and $n \in \N$, let $T_n = T_n(\phi,U)$ and note that $T_n$ is a subgroup of $G$ as the latter is tqh. Then $U \leq T_n \cap (UH) \leq T_n$, so that
\begin{equation}\label{two:indexes}
[T_n : U] = [T_n: T_n \cap (UH)] [T_n \cap (UH):U],
\end{equation}
and we study separately the two indices in the right hand side of the above equation.
	
Let $a = [T_n: T_n \cap (UH)] $ and consider the following Hasse diagram in the lattices of subgroups of $G$ and of $G/H$.
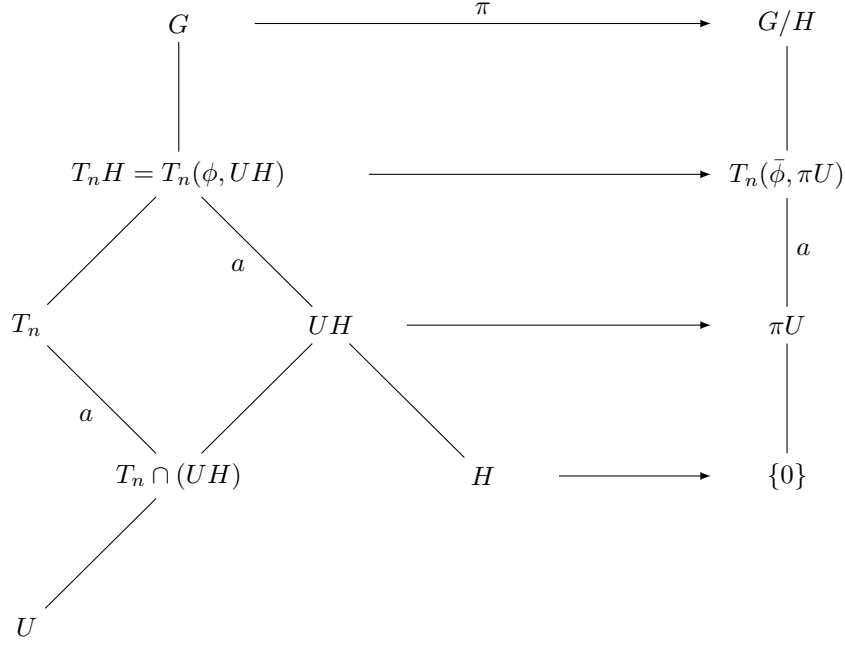
\begin{center}
\begin{tikzpicture}
\node[above] at (4,0) {$\pi$};
\node(G) at (0,0) {$G$};											\node(G/H) at (8,0) {$G/H$};
\node (TnH) at (0,-2) {$T_n  H = T_n(\phi, UH)$};			\node(piTn) at (8,-2) {$T_n(\bar \phi, \pi U)$};
\node[below left] at (1,-3) {$a$};								\node [right] at(8,-3) {$a$};
\node(UH)      % [below right of=Tn+H]
at (2,-4) {$UH$};										\node(piU) at (8,-4) {$\pi U$};
\node(Tn)      %[below left of=Tn+H]
at (-2,-4) {$T_n$};
\node[below left] at (-1,-5) {$a$};
\node(inters)     % [below left of=U+H]
at (0,-6)  {$T_n \cap (UH)$};
\node(U)     % [below left of=inters]
at (-2,-8)     {$U$};
\node(H)    %  [below right of=U+H]
at (4,-6)  {$H$};										\node(o) at (8,-6) {$\{0\}$};
		%\node(Tn intH)      [below left of=H]       {$T_n \cap H$};
		%\node(UintH)      [below right of=U]       {$U\cap H$};
\draw(G) -- (TnH);
\draw(TnH)       -- (UH);
\draw(TnH)       -- (Tn);
\draw(Tn)       -- (inters);
\draw(UH)       -- (inters);
\draw(inters)       -- (U);
\draw(UH)       -- (H);
\draw(G/H) -- (piTn) -- (piU) -- (o);
\draw[-latex](1,0)--(7,0);
\draw[-latex](2.5,-2)--(7,-2);
\draw[-latex](3,-4)--(7,-4);
\draw[-latex](5,-6)--(7,-6);
\end{tikzpicture}
\end{center}
	
Now we show that $T_n H = T_n(\phi, UH)$. Indeed, since $H$ is normal in $G$ it follows that $T_n H \subseteq T_n(\phi, UH)$. For the converse containment, first note that $H$ is $\phi$-invariant. So, using also the normality of $H$ we have  \[T_n(\phi, UH)=(UH) \phi(UH)\cdots \phi^{n-1}(UH)\subseteq (UH) \phi(U)H\cdots \phi^{n-1}(U)H=T_nH.\]
The equality $T_n H = T_n(\phi, UH)$ implies  that
\[a= [T_n  H : UH] = [ T_n(\phi, UH) : UH].\]
Moreover, both $T_n(\phi, UH)$ and $UH$ contain $H = \ker \pi$, so considering their images in $G/H$ we have
$[ T_n(\phi, UH) : UH]= [\pi ( T_n(\phi, UH) ) : \pi(UH)]$, and the latter coincides with $[T_n(\bar \phi, \pi U ):  \pi U]$ by \cite[Lemma 4.2]{GBST}.
	
To study the second index in the right hand side of Equation \eqref{two:indexes}, let $b = [T_n \cap (UH):U]$. As $T_n \cap (UH) = (T_n\cap H) U$ by the modular law, we have
\[b = [(T_n\cap H) U : U] = [T_n \cap H:U\cap H].\]
Since $U\cap H$ is a compact subgroup of the tqh group $G$, we deduce that  $T_n(\phi, U\cap H)\leq G.$ The $\phi$-invariance of $H$ implies that $T_n(\phi, U\cap H)$ is a subgroup of $T_n \cap H$.
So, we obtain \[b =  [T_n \cap H:U\cap H] \geq [T_n(\phi, U\cap H):U\cap H].\]
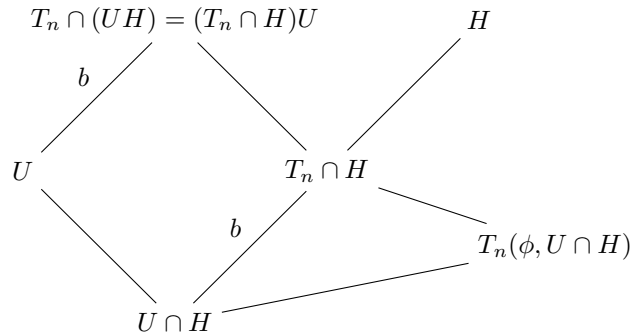
\begin{center}
\begin{tikzpicture}
\node(TncapH U) at(0,0)     %[below of=inters]
{$%\mkern-70mu
T_n \cap (UH) = (T_n\cap H)U$};		\node[above left] at (-1,-1) {$b$};
\node(U)     % [below left of=TncapH +U]
at (-2,-2)     {$U$};
\node(Tn intH)     % [below right of=TncapH +U]
at (2,-2)     {$T_n \cap H$};
\node(H)    % [above right of=Tn intH]
at (4,0)    {$H$};% [right of=TncapH +U]       {$H$};
\node (Tnint) at (5,-3) {$T_n(\phi, U\cap H)$};
\node(UintH)    %  [below right of=U]
at (0,-4)    {$U\cap H$};				\node[above left] at (1,-3) {$b$};
\draw(TncapH U)       -- (U);
\draw(TncapH U)       -- (Tn intH);
\draw(H)       -- (Tn intH);
\draw(U)       -- (UintH);
\draw(Tn intH)       -- (UintH);
\draw(Tn intH)       --(Tnint)--(UintH);
\end{tikzpicture}
\end{center}
Finally, from Equation \eqref{two:indexes} and the above discussion it follows that
\begin{equation*}
[T_n : U] \geq [T_n( \bar \phi, \pi U):  \pi U] [T_n(\phi, U\cap H):U\cap H].
\end{equation*}
Applying $\log$, dividing by $n$, and taking the limit for $n \rightarrow \infty$ we conclude that, for every $U\in \mathcal B(G)$,
\begin{equation} \label{first:half:add:thm:eq1}
H_{alg} (\phi, U) \geq H_{alg} (\bar \phi, \pi U) + H_{alg} (\phi \restriction_H, U \cap H).
\end{equation}
	
Let $U_1, U_2 \in \mathcal B(G)$, and $U = U_1 U_2 \in \mathcal B(G)$. Then $\pi U \geq \pi U_1$ are elements of $\mathcal B(G/H)$, and $U \cap H \geq U_2 \cap H$ are in $\mathcal B(H)$, so that
\begin{gather}	
H_{alg} (\bar \phi, \pi U) \geq H_{alg} (\bar \phi, \pi U_1), \label{first:half:add:thm:eq2}\\
H_{alg} (\phi \restriction_H, U \cap H) \geq H_{alg} (\phi \restriction_H, U_2 \cap H). \label{first:half:add:thm:eq3}
\end{gather}
From Equations \eqref{first:half:add:thm:eq1}, \eqref{first:half:add:thm:eq2} and \eqref{first:half:add:thm:eq3}, it follows that
\begin{equation*}
h_{alg} (\phi) \geq H_{alg} (\phi, U) \geq H_{alg} (\bar \phi, \pi U_1) + H_{alg} (\phi \restriction_H, U_2 \cap H).
\end{equation*}
Taking the suprema over $U_1, U_2 \in \mathcal B(G)$, we conclude by applying Corollary \ref{cor:halg:induced:maps}.
\end{proof}

Recall that  the discrete compactly covered tqh groups  are exactly the torsion quasihamiltonian groups.

\begin{proposition}\label{lem:dirsum}
Let $G$ be a torsion quasihamiltonian group, $\phi\in \End(G)$ and $\bar{\phi}: G/G'\to G/G'$ be the induced map.
Consider also $\phi_p=\phi\restriction_{G_p}$ and the induced map $\overline{\phi_p}:G_p/G_p'\to G_p/G_p'$.
Then,
\ben
\item $h_{alg}(\phi)=\sum_{p\in \mathbb P}h_{alg}(\phi_p)$,
\item $h_{alg}(\phi\upharpoonright_{G'})=\sum_{p\in \mathbb P}h_{alg}(\phi{_p}\upharpoonright_{G_p^{'}})$,
\item $h_{alg}(\bar \phi)=\sum_{p\in \mathbb P}h_{alg}(\overline{\phi_p})$.
\een
\end{proposition}
\begin{proof}
We just prove (1). The proofs of (2) and (3) are similar.

First, we show that $h_{alg}(\phi)\leq\sum_{p\in \mathbb P}h_{alg}(\phi_p)$. Fix a finite subgroup $F$ of  $G$.  By Lemma \ref{lem:dir}(1), there exist finitely many primes $p_1,p_2,\ldots,p_k$ such that $F=F_{p_1}\oplus F_{p_2}\oplus\cdots \oplus F_{p_k}$, where $F_{p_i}$ is a finite subgroup of  $G_{p_i}$ for every $i\in \{1,2,\cdot\cdot\cdot,k\}$. As $G$ is a tqh group, $T_n(\phi, F)$ is a subgroup of $G$. Moreover,  Remark \ref{rem:per}(2) implies  that \[T_n(\phi, F)=T_n(\phi, F_{p_1})\cdot T_n(\phi, F_{p_2})\cdot\ldots\cdot T_n(\phi, F_{p_n}).\] It follows that
\[|T_n(\phi, F)|\leq |T_n(\phi, F_{p_1})|\cdot|T_n(\phi, F_{p_2})|\cdot\ldots\cdot|T_n(\phi, F_{p_n})|. \]
Taking log, dividing by $n$ and letting the limit for $n\to\infty$ we deduce that
\[H_{alg}(\phi,F)\leq H_{alg}(\phi,F_{p_1})+H_{alg}(\phi,F_{p_2})+\ldots +H_{alg}(\phi,F_{p_k})\leq\sum_{p\in \mathbb P}h_{alg}(\phi_p) \]
 Applying Equation \eqref{eq:hdis} we conclude that $h_{alg}(\phi)\leq\sum_{p\in \mathbb P}h_{alg}(\phi_p).$

Now we prove that $h_{alg}(\phi)\geq\sum_{p\in \mathbb P}h_{alg}(\phi_p)$. If there exists some prime $p$ such that $h_{alg}(\phi_p)=\infty$, then $h_{alg}(\phi)=\infty$ by Proposition \ref{monotonicity}(1). So, one can assume that $h_{alg}(\phi_p)<\infty$ for every prime $p$.

Suppose that there exists an infinite subset $A\subset \mathbb P$ such that $h_{alg}(\phi_p)\neq0$ for every $p\in A$. For every $m\in \N$, choose $p_0\in A$ such that $\log p_0> m$. By Remark \ref{p-group}, $h_{alg}(\phi_{p_0})\geq \log p_0$, so $h_{alg}(\phi_{p_0})> m$.
Using Proposition \ref{monotonicity}(1) again, we get that $h_{alg}(\phi)\geq h_{alg}(\phi_{p_0})> m$ for every $m\in \N$. It follows that $h_{alg}(\phi)=\infty\geq\sum_{p\in \mathbb P}h_{alg}(\phi_p)$.

Consequently, one can assume that $h_{alg}(\phi_p) = 0$ for $p\in \mathbb P\setminus B$, where $B$ is a finite subset of $\mathbb P$. So, $\sum_{p\in \mathbb P}h_{alg}(\phi_p)=\sum_{p\in B}h_{alg}(\phi_p)$. Without loss of generality, let $B=\{p_1, p_2\}$ and $H=G_{p_1}\cdot G_{p_2}.$   As $h_{alg}(\phi) \geq h_{alg}(\phi\upharpoonright_H)$ and $h_{alg}(\phi\upharpoonright_H)\geq h_{alg}(\phi_{p_1})+h_{alg}(\phi_{p_2})$ by Proposition \ref{first:half:add:thm}, it follows that \[h_{alg}(\phi) \geq h_{alg}(\phi_{p_1})+h_{alg}(\phi_{p_2}),\] as needed.
\end{proof}

Under the additional assumption of  $G$ being an FC-group, Proposition \ref{lem:dirsum}  gives the next corollary.

\begin{corollary}\label{prop:metaderi}
Let $G$ be a torsion quasihamiltonian FC-group, $\phi \in \End(G)$, and  let $\bar \phi : G/G' \to G/G'$ be the induced map. Then $h_{alg} (\phi) = h_{alg} (\bar \phi)$, and $h_{alg} (\phi \restriction_{G'})=0$. In particular, $\AT(G,G',\phi)$ holds.
\end{corollary}
\begin{proof}
Assume first that $G$ is a $p$-group. As the derived subgroup $G'$ is finite in a quasihamiltonian FC $p$-group by Proposition \ref{prop:pquafc}, then $h_{alg} (\phi \restriction_{G'})=0$, and $h_{alg}(\phi)= h_{alg}(\bar \phi)$ by Lemma \ref{lem:finite}(2).

For the general case, we let $\phi_p=\phi\restriction_{G_p}$ and $\overline{\phi_p}:G_p/G_p'\to G_p/G_p'$  for every $p\in \mathbb P$. From the previous step and using Proposition \ref{lem:dirsum} we obtain
\[h_{alg}(\phi)=\sum_{p\in \mathbb P}h_{alg}(\phi_p)=\sum_{p\in \mathbb P}h_{alg}(\overline{\phi_p})=h_{alg}(\bar \phi)\leq  h_{alg} (\bar \phi) + h_{alg} (\phi \restriction_{G'}).\]

By Proposition \ref{first:half:add:thm}, we also have $h_{alg}(\phi) \geq  h_{alg} (\bar \phi) + h_{alg} (\phi \restriction_{G'})$, so we deduce $h_{alg} (\phi \restriction_{G'})=0$.
\end{proof}

Using Proposition \ref{LFcoralg}, and considering some additional assumptions on the subgroup $H$ of $G$, we prove the converse inequality proved in Proposition \ref{first:half:add:thm}, thus obtaining AT$(G, H, \phi)$ in this case.
\begin{theorem}\label{addthm}
Let $G$ be a compactly covered locally compact tqh group, $\phi\in \End(G)$, $H$ a closed normal $\phi$-stable subgroup of $G$ with $\ker\phi\leq H$, and $\bar \phi : G/H \to G/H$ the induced map. Then AT$(G, H, \phi)$ holds.
\end{theorem}
\begin{proof}
In view of Proposition \ref{first:half:add:thm}, it only remains to prove that $h_{alg} (\phi) \leq h_{alg} (\bar \phi) + h_{alg}(\phi\restriction_H~).$

Let $O$ be an arbitrary open subgroup of $G$ such that $\phi^{-1}O\leq O$ and $[O:\phi^{-1}O]< \infty$.
Since $H$ is normal in $G$, we have
\begin{equation}\label{twoind}
[O:\phi^{-1}O] = [O : (\phi^{-1}O)(O\cap H)] \cdot [(\phi^{-1}O)(O\cap H) : \phi^{-1}O].
\end{equation}
\begin{center}
\begin{tikzpicture}
\node(O) at(0,0)     {$O$};
\node(sum)    at (0,-2)     {$(\phi^{-1}O)(O\cap H)$};
\node(phimO)  at (-2,-4)     {$\phi^{-1}O$};
\node(OintH)    at (2,-4)    {$O \cap H$};
\node (int) at (0,-6) {$\phi^{-1}O \cap H$};
\node(H)   at (4,-2)    {$H$};	
\draw(O)  -- (sum) -- (phimO) -- (int) -- (OintH) -- (sum);
\draw(H)   -- (OintH);
\end{tikzpicture}
\end{center}

First observe that since $H$ is $\phi$-stable with  $\ker\phi\leq H$, then we also have $\phi^{-1}H=H$, so
$$\phi^{-1}(O)\cap H = \phi^{-1}O\cap \phi^{-1}H = \phi^{-1}(O\cap H).$$
Then, computing the second index in the right hand side of Equation \eqref{twoind} we obtain
\begin{equation*}
[(\phi^{-1}O)(O\cap H) : \phi^{-1}O] = [O\cap H:\phi^{-1}(O)\cap H] =[O\cap H:\phi^{-1}(O\cap H)].
\end{equation*}
Note that $H$ is a compactly covered locally compact tqh group, and having an open normal subgroup $O\cap H$ such that $\phi^{-1}(O\cap H)\leq O\cap H$, and $[O\cap H:\phi^{-1}(O\cap H)]<\infty$ by  Equation \eqref{twoind}. By Proposition \ref{LFcoralg},
\begin{equation}\label{2:add:th:eq:1}
h_{alg} (\phi \restriction_H) \geq \log [O\cap H:\phi^{-1}(O\cap H)].
\end{equation}
To compute the first index in the right hand side of Equation \eqref{twoind}, first note that $$(\phi^{-1}O)(O\cap H) = (\phi^{-1}(O) H )\cap O$$ by the modular law. Then, chasing the diagram
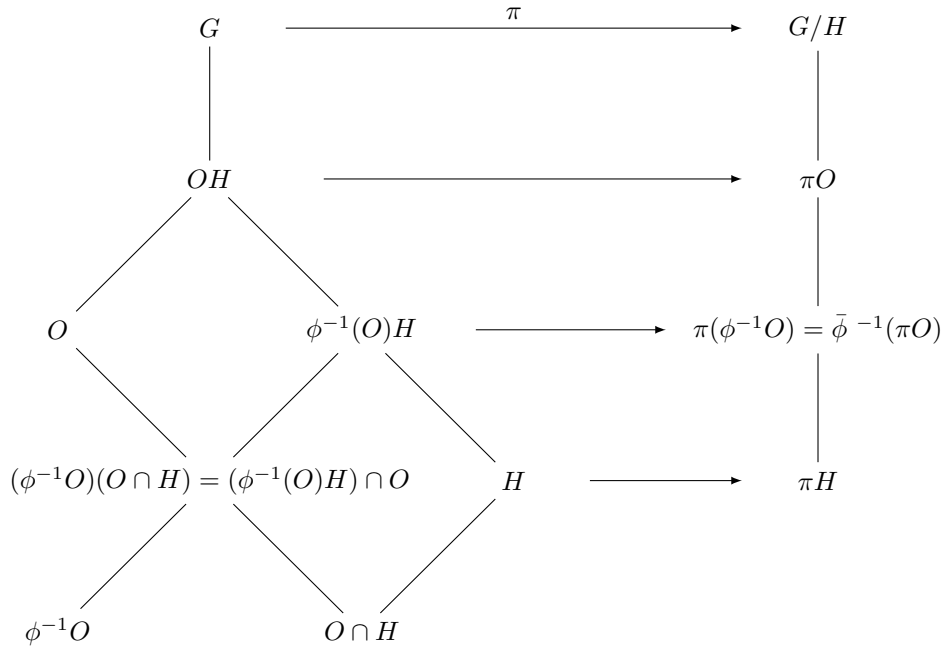
\begin{center}
\begin{tikzpicture}
\node(O) at(-2,0)     {$O$};		
\node(sum)    at (0,-2)  {$(\phi^{-1}O)(O\cap H) = (\phi^{-1}(O) H )\cap O$};
\node(phimO)  at (-2,-4)     {$\phi^{-1}O$};
\node(OintH)    at (2,-4)    {$O \cap H$};
\node(H)   at (4,-2)    {$H$};										\node(zero) at (8,-2) {$\pi H$};
\node(phimO+H) at (2,0) {$\phi^{-1}(O) H$};					\node(piphimO) at (8,0) {$\pi (\phi^{-1}O) = {\bar \phi}^{\ -1}(\pi O)$};
\node(O+H) at (0,2) {$OH$};											\node(piO+H) at (8,2) {$\pi O$};
\node(G) at (0,4) {$G$};												\node(G/H) at (8,4) {$G/H$};
\draw(O+H) -- (O)  -- (sum) -- (OintH) --(H) -- (phimO+H) -- (sum) -- (phimO);
\draw(phimO+H) --  (O+H) -- (G);
\draw(G/H) -- (piO+H) -- (piphimO) -- (zero);
				\node[above] at (4,4) {$\pi$};
\draw[-latex](1,4)--(7,4);
\draw[-latex](1.5,2)--(7,2);
\draw[-latex](3.5,0)--(6,0);
\draw[-latex](5,-2)--(7,-2);
\end{tikzpicture}
\end{center}
of the subgroups of $G$ and $G/H$, one can easily verify that
\begin{multline*}
[O:(\phi^{-1}O)(O\cap H)] = [O : (\phi^{-1}(O) H )\cap O] = [OH:\phi^{-1}(O)H] = \\ = [\pi(OH) : \pi ( \phi^{-1}(O)H )] = [\pi O:\pi (\phi^{-1}O)] = [\pi O: {\bar \phi}^{\ -1}(\pi O)].
\end{multline*}

By Fact \ref{lem:s2q}(1), $G/H$ is a compactly covered locally compact tqh group. As $\pi O$ is an open subgroup of $G/H$, such that ${\bar \phi}^{\ -1}(\pi O)\leq \pi O$ and $[\pi O:{\bar \phi}^{\ -1}(\pi O)]<\infty$ by Equation \eqref{twoind}, Proposition \ref{LFcoralg} applied to $G/H$ and to $\bar \phi$ implies that
\begin{equation}\label{2:add:th:eq:2}
%\log[O:\phi^{-1}O+(O\cap H)]=\log [\pi O:{\bar \phi}^{\ -1}(\pi O)]\leq  h_{alg} (\bar \phi).
h_{alg} (\bar \phi) \geq \log [\pi O:{\bar \phi}^{\ -1}(\pi O)].
\end{equation}
Summing up Equation \eqref{2:add:th:eq:1} and Equation \eqref{2:add:th:eq:2}, and using Equation \eqref{twoind}, we obtain
\[h_{alg} (\bar \phi)+h_{alg} (\phi \restriction_H) \geq \log [O:\phi^{-1}O].\]
By the arbitrariness of $O$ and applying Proposition \ref{LFcoralg} to $G$ we conclude that $h_{alg} (\bar \phi) + h_{alg} (\phi \restriction_H) \geq h_{alg} (\phi)$.
\end{proof}

The next result shows that the Addition Theorem holds for automorphisms of compactly covered locally compact tqh groups (e.g., discrete torsion quasihamiltonian groups).  It is an immediate corollary  of Theorem \ref{addthm}.
\begin{corollary}\label{cor:foraut}
	If $G$ is a compactly covered locally compact tqh group, $\phi\in \Aut(G)$ and $H$ is a normal $\phi$-stable subgroup of $G$, then  \begin{equation*}
	h_{alg} (\phi) =h_{alg} (\bar \phi) + h_{alg} (\phi \restriction_H).\end{equation*}
\end{corollary}

In fact, when we compute the algebraic entropy of a topological \emph{automorphism}  of a compactly covered locally compact tqh group, we may assume that it is also totally disconnected.

\begin{corollary}\label{c(G)}\
Let $G$ be a compactly covered locally compact tqh group, and $\phi \in \End(G)$ be such that $c(G)$ is $\phi$-stable and contains $\ker \phi$.  Then  $h_{alg} (\phi) =h_{alg} (\bar \phi)$, where $\bar \phi : G/c(G) \to G/c(G)$ is the induced map.
	
In particular, if $\phi\in \Aut(G)$,  then $c(G)$ is $\phi$-stable and $h_{alg} (\phi) =h_{alg} (\bar \phi)$.
\end{corollary}
\begin{proof}
By Fact \ref{Mukhin}, if $G$ is not totally disconnected, then it is abelian. So, we may assume that  $G$ is compactly covered locally compact abelian group. In particular, $G$ is strongly compactly covered by \cite[Corollary 2.1]{GBST} so \cite[Corollary 7.7]{GBST} applies.
\end{proof}

\subsection{Quasihamiltonian torsion FC-groups}
Recall that a group  is {\em virtually nilpotent} if it contains a nilpotent subgroup having finite index.

\begin{proposition}
If a finitely generated metabelian group $G$ satisfies $\AT(G,G', id_G)$, then $G$ is virtually nilpotent.
\end{proposition}
\begin{proof}
As $G'$ and $G/G'$ are abelian groups, by \cite[Example 2.5]{DGBabelian}, we obtain $h_{alg}(id_G\upharpoonright G')=0$ and  $h_{alg}(\overline{id_G})=0$, where $\overline{id_G}\in \End(G/G')$ induced by $id_G$.
It follows that $h_{alg}(id_G)=0$ by AT$(G,G', id_G)$.
Hence, $G$ has either polynomial growth or intermediate growth by \cite[Proposition 5.3.13(a)]{DG-islam}. Moreover, $G$ has polynomial growth by Milnor-Wolf's Theorem (see \cite{MW}).
So, the finitely generated group $G$ is virtually nilpotent by Gromov Theorem (see \cite{Gro}).
\end{proof}

\begin{proposition}\label{prop:meta}
Let $G$ be a metabelian group, $\phi\in \End(G)$, $H$ a normal $\phi$-invariant subgroup of $G$, and $\bar \phi : G/H \to G/H$ the induced map.

If $\AT(G, G', \phi)$, $\AT(H, H', \phi\upharpoonright_H)$, and $\AT(G/H, (G/H)', \bar\phi)$ hold, then $\AT(G, H, \phi)$ holds.
\end{proposition}
\begin{proof}
By $\AT(G, G', \phi)$, we deduce that
\begin{equation}\label{meta1}
h_{alg}(\phi)=h_{alg}(\phi\upharpoonright_{G'})+h_{alg}(\widetilde{\phi}),
\end{equation}
where $\widetilde{\phi} : G/G'\to G/G'$ is the  map induced by $\phi$.

As $G'$ is an abelian group, and $G'\cap H$ is a $\phi$-invariant subgroup of $G'$, we get by the Addition Theorem (see \cite[Theorem 1.1]{DGBabelian})
\begin{equation}\label{meta2}
h_{alg}(\phi\upharpoonright_{G'})=h_{alg}(\phi\upharpoonright_{G'\cap H})+h_{alg}(\widetilde{\phi\upharpoonright_{G'}}),
\end{equation}
where $\widetilde{\phi\upharpoonright_{G'}} \in \End(G'/G'\cap H)$ is the  map induced by $\phi\upharpoonright_{G'}$.

Similarly, since $G/G'$ is an abelian group, and $HG'/ G'$ is a $\widetilde{\phi}$-invariant subgroup of $G/G'$, we obtain
\begin{equation}\label{meta3}
h_{alg}(\widetilde{\phi})=h_{alg}(\widetilde{\phi}\upharpoonright_{HG'/G'})+h_{alg}(\overline{\widetilde{\phi}}),
\end{equation}
where $\overline{\widetilde{\phi}} \in \End((G/G')/ (HG'/G'))$ is the  map induced by $\widetilde{\phi}$.

Hence, by Equations (\ref{meta1}), (\ref{meta2}) and (\ref{meta3}), we have
\begin{equation}\label{meta4}
h_{alg}(\phi)=h_{alg}(\phi\upharpoonright_{G'\cap H})+h_{alg}(\widetilde{\phi\upharpoonright_{G'}})+h_{alg}(\widetilde{\phi}\upharpoonright_{HG'/G'})+h_{alg}(\overline{\widetilde{\phi}}).
\end{equation}

\noindent
\textbf{Claim 1} $h_{alg}(\phi\upharpoonright_H)=h_{alg}(\phi\upharpoonright_{G'\cap H})+h_{alg}(\widetilde{\phi}\upharpoonright_{HG'/G'})$.
\begin{proof}

By $\AT(H, H', \phi\upharpoonright_{H})$, we deduce that
\begin{equation}\label{meta6}
h_{alg}(\phi\upharpoonright_{H})=h_{alg}(\phi\upharpoonright_{H'})+h_{alg}(\widetilde{\phi\upharpoonright_{H}}),
\end{equation}
where $\widetilde{\phi\upharpoonright_{H}} : H/H'\to H/H'$ is the  map induced by $\phi\upharpoonright_{H}$.

As $G'\cap H$ is abelian, and $H'$ is a $\phi$-invariant subgroup of $G'\cap H$, we deduce by the Addition Theorem that
\begin{equation}\label{meta5}
h_{alg}(\phi\upharpoonright_ {G'\cap H})=h_{alg}(\phi\upharpoonright_{H'})+h_{alg}(\xi),
\end{equation}
where $\xi : (G'\cap H)/H'\to (G'\cap H)/H'$ is the  map induced by $\phi\upharpoonright_{G'\cap H}$.

Hence, to prove Claim 1, it suffices to show that
\begin{equation}\label{meta8}
h_{alg}(\widetilde{\phi\upharpoonright_{H}}) = h_{alg}(\xi)+ h_{alg}(\widetilde{\phi}\upharpoonright_{HG'/G'}).
\end{equation}

First observe that  $H/H'$ is abelian, and $(H\cap G')/ H'$ is a $\widetilde{\phi\upharpoonright_{H}}$-invariant subgroup of $H/H'$. Moreover, $\xi= \widetilde{\phi\upharpoonright_{H}}\upharpoonright_{((H\cap G')/H'}$, so by the Addition Theorem for abelian groups we obtain
\begin{equation}\label{meta7}
h_{alg}(\widetilde{\phi\upharpoonright_{H}})=h_{alg}(\xi)+h_{alg}(\varphi),
\end{equation}
where $\varphi \in \End((H/H')/ ((H\cap G')/H'))$ is the  map induced by $\widetilde{\phi\upharpoonright_{H}}$.

Therefore,  to prove Equation (\ref{meta8}) it suffices to show that $h_{alg}(\varphi)=h_{alg}(\widetilde{\phi}\upharpoonright_{HG'/G'})$, and this equality follows from  the Invariance under conjugation property.
\end{proof}

\noindent
\textbf{Claim 2} $h_{alg}(\bar\phi)=h_{alg}(\widetilde{\phi\upharpoonright_{G'}})+h_{alg}(\overline{\widetilde{\phi}})$.
\begin{proof}
Let $M=(G/H)/(G/H)'$, $N=(G'H/H)/(G/H)'$ and $\eta\in \End(N)$ be the map induced by $\bar\phi$. As $N$ is abelian, and $N$ is a $\eta$-invariant subgroup of $M$, we deduce that $\AT(M, N, \eta)$ holds.

Let $\gamma=\bar\phi \upharpoonright G'H/H$. Since $(G/H)'$ is a $\gamma$-invariant subgroup of the abelian group $G'H/H$, by the Addition Theorem, we have $\AT((G'H/H, (G/H)', \gamma)$.

In addition, $\AT((G/H, (G/H)', \bar\phi)$ holds by our assumption. Using similar arguments to those appearing in the proof of  Claim 1, one can show that the properties $\AT(M, N, \eta)$, $\AT((G'H/H, (G/H)', \gamma)$ and $\AT((G/H, (G/H)', \bar\phi)$ imply that $h_{alg}(\bar\phi)=h_{alg}(\widetilde{\phi\upharpoonright_{G'}})+h_{alg}(\overline{\widetilde{\phi}})$.
\end{proof}

Claim 1, Claim 2 and Equation (\ref{meta4}) complete the proof of Proposition \ref{prop:meta}, i.e.,\
\[
h_{alg}(\phi)=h_{alg}(\phi\upharpoonright_H)+h_{alg}(\bar\phi).\qedhere
\]
\end{proof}

Recall that quasihamiltonian groups are metabelian by Fact \ref{quasi}(2).
As the class of quasihamiltonian torsion FC-groups is stable under taking subgroups and Hausdorff quotients, Corollary \ref{prop:metaderi} and Proposition \ref{prop:meta} imply the following Addition Theorem for this class of groups. %To this end, recall that quasihamiltonian groups are metabelian by Fact \ref{quasi}(2).
\begin{theorem}\label{torquasiFC}
Let $G$ be a quasihamiltonian torsion FC-group, $\phi \in \End(G)$, $H$ a normal $\phi$-invariant subgroup of $G$, and $\bar \phi : G/H \to G/H$ the induced map. Then
\begin{equation*}
h_{alg} (\phi) = h_{alg} (\bar \phi) + h_{alg} (\phi \restriction_H).
\end{equation*}
\end{theorem}

As a particular case of the above theorem, we immediately obtain the following result.

\begin{corollary}\label{thm:ATham}
Let $G$ be a Hamiltonian group, $\phi\in \End(G)$, $H$ a $\phi$-invariant subgroup of $G$, and $\bar \phi : G/H \to G/H$ the induced map. Then AT$(G, H, \phi)$ holds.
\end{corollary}

\section{Open questions and concluding remarks}\label{Open questions and concluding remarks}
By Theorem \ref{torquasiFC} we know that the Addition Theorem holds for endomorphisms of torsion quasihamiltonian FC-groups. Can we omit the assumption `FC' in the hypotheses of Theorem \ref{torquasiFC}? In other words:
\begin{question}\label{withoutFC}
Does the Addition Theorem hold for every endomorphism of a torsion quasihamiltonian group?
\end{question}

By Proposition \ref{lem:dirsum}, to give a positive answer to Question \ref{withoutFC}, it is sufficient to check that the Addition Theorem holds for quasihamiltonian $p$-groups. On the other hand, one can try to construct  a counter example. Studying the endomorphisms of the group from Example \ref{example:nonFC}, which is  quasihamiltonian $3$-group but not FC, could be a good starting point.

\vskip 0.3cm
 Recall that the discrete strongly compactly covered
 tqh groups are exactly the quasihamiltonian torsion FC-groups.  In view of  Theorem \ref{torquasiFC} we ask:

\begin{question}\label{topver}
Does the Addition Theorem hold for every endomorphism of a strongly compactly covered %non-abelian
 tqh group?
\end{question}

A topologically Hamiltonian group is a topological group in which every closed subgroup is normal. Strunkov \cite{ST} proved that  if $G$ is a locally compact topologically Hamiltonian group, then $G \cong Q_8\times B \times D$, where $Q_8$ is the quaternion group of order $8$, $B$ is a locally compact Boolean group and $D$ is a  locally compact torsion abelian group with all its elements of odd order. In particular, every locally compact topologically Hamiltonian group is a strongly compactly covered tqh group.

One can study Question \ref{topver} in the particular case of the locally compact topologically Hamiltonian groups.

\begin{question}\label{topver2}
Does the Addition Theorem hold for every endomorphism of a locally compact topologically Hamiltonian group?
\end{question}

Note that a positive answer to Question  \ref{topver2} would provide an extension of Corollary \ref{thm:ATham}.

\subsection*{Acknowledgments}  We thank Professor Herfort for sharing with us some of the secrets of the tqh groups.
The first-named author takes this opportunity to thank Professor Dikranjan for his generous hospitality and support.
The second and third-named authors have been supported by Programma SIR 2014 by MIUR, project GADYGR, number RBSI14V2LI, cup G22I15000160008.

	\end{document}